\documentclass{article}
\usepackage[utf8]{inputenc}

\usepackage{graphicx,graphics}

\usepackage{rotating}
\usepackage[twoside]{geometry}
\usepackage{epsfig}
\usepackage{indentfirst}
\usepackage[usenames,dvipsnames,svgnames,table]{xcolor}
\usepackage{cmap}

\usepackage{graphicx,graphics}

\usepackage{hyperref}
\hypersetup{
pdftitle={Slow decay},    
pdfauthor={},    
pdfnewwindow=true,      
colorlinks=true,      
linkcolor=black,         
citecolor=blue}

\usepackage{amssymb,amsfonts,amstext,amsthm}
\usepackage{mathrsfs}
\usepackage[intlimits]{amsmath}
\newcommand{\OBSI}{\begin{remark}\begin{rm}}
\newcommand{\OBSF}{\end{rm}\end{remark}}
\newcommand{\DEFI}{\begin{definition}\begin{rm}}
\newcommand{\DEFF}{\end{rm}\end{definition}}

\newcommand{\be}{\begin{eqnarray}}
\newcommand{\en}{\end{eqnarray}}
\newcommand{\bee}{\begin{eqnarray*}}
\newcommand{\ene}{\end{eqnarray*}}

\DeclareMathOperator*{\essinf}{ess.inf}

\newtheorem{definition}{\bf Definition}[section]
\newtheorem{proposition}{\bf Proposition}[section]

\newtheorem{lemma}{\bf Lemma}[section]
\newtheorem{theorem}{\bf Theorem}[section]

\newtheorem{remark}{Remark}[section]

\title{Slow dynamics for self-adjoint semigroups and unitary evolution groups}
\author{M. Aloisio\thanks{Corresponding author.}, \, S. L. Carvalho, \, C. R. de Oliveira \, and \, G. Santana}
\date{September 2022}

\begin{document}

\maketitle

\begin{abstract} We obtain slow dynamics  for self-adjoint semigroups and unitary evolution groups. For semigroups, the slow dynamics is for orbits, and for the average  return probability in the case of unitary evolution groups. We present an application to the quantum dynamics of purely absolutely continuous systems.
\end{abstract}    


\section{Introduction}

\subsection{Contextualization}

 \noindent The existence of orbits of operator semigroups that converge to zero arbitrarily slowly has been studied by many authors in the last two decades (see \cite{ACdeOCategory,ACdeOscalesemigroup,Muller0,Muller1,Muller2,Muller} and references therein). Pioneering works were established by M\"uller in the discrete case  \cite{Muller0,Muller1,Muller2}. Namely, given any $\epsilon > 0$ and a sequence of real numbers $(a_n)_{n \geq 1}$ satisfying $|a_n| \leq 1$ for all $n \geq 1$ and $\displaystyle\lim_{n \to \infty} a_n = 0$, a known result in~\cite{Muller0} states that if $T$ is a bounded operator on a complex Banach space $X$ with spectral radius equal to~1, then there exists a normalized $\psi \in X$  such that
\[\Vert T^n\psi\Vert_{X} \geq (1-\epsilon)|a_n|,  \quad \forall n \geq 1.\]

With respect to the continuous case, M\"uller and Tomilov~\cite{Muller} have established several analogous results. E.g., let $(T(t))_{t\geq 0}$ be a weakly stable $C_0$-semigroup on a Hilbert space~$\mathcal{H}$ (i.e, it converges weakly to zero as $t \to \infty$)  such that 
\[\lim_{t \to \infty} \frac{\ln \|T(t)\|_{\mathcal{B}(\mathcal{H})}}{t} = 0.\] 
Let $g: {\mathbb{R}}_+ \longrightarrow (0,\infty)$ be a bounded function such that $\displaystyle\lim_{t \to \infty} g(t) = 0$ and let $\epsilon>0$. Then, there exists $\psi \in \mathcal{H}$  so that $\Vert \psi\Vert_{\mathcal{H}} < \displaystyle\sup_{t \geq 0} \{g(t)\} + \epsilon$ and
\begin{equation}\label{contexeq1}
\vert \langle T(t)\psi,\psi \rangle \vert > g(t), \quad \forall t \geq 0.
\end{equation}

Still with respect to the continuous case, the first three authors of this work have explored the  above result~\cite{ACdeOCategory,ACdeOscalesemigroup} to show  that the decaying rates of orbits of semigroups, which are stable but not exponentially stable, typically in Baire's sense, depend on sequences of time going to infinity. In the case of self-adjoint semigroups, it was also show~\cite{ACdeOscalesemigroup} that there is an explicit relation between the dynamics of the semigroup and  local scale spectral properties of its generator. We recall some details. 
   
Let $\mu$ be  a finite (positive) Borel measure on $\mathbb{R}$ and $B(w,\epsilon)=(w-\epsilon,w+\epsilon)$. The pointwise lower and upper local scaling exponents of $\mu$  at $w \in \mathbb{R}$ are defined, respectively, by  
\[d_\mu^-(w) := \liminf_{\epsilon \downarrow 0} \frac{\ln \mu (B(w,\epsilon))}{\ln \epsilon} \quad{\rm  and }\quad d_\mu^+(w) := \limsup_{\epsilon \downarrow 0} \frac{\ln \mu (B(w,\epsilon))}{\ln \epsilon}\,,\]
if, for all  $\epsilon>0$,  $\mu(B(w,\epsilon))> 0$; $d_\mu^{\mp}(w) := \infty$, otherwise.

\begin{proposition}[Proposition 2.2 in \cite{ACdeOscalesemigroup}]\label{decpolproposition} Let~$T$ be a negative self-adjoint  operator, i.e., $T \leq 0$, and let $\psi \in \mathcal{H}$, with $\psi \not = 0$. Then,
\begin{equation*}
d_{\mu_\psi^{T}}^+(0) = -\liminf_{t \to \infty} \frac{\ln  \|e^{tT}\psi\|_{{\mathcal{H}}}^2}{\ln t}  \quad\textrm{and}\quad   d_{\mu_\psi^{T}}^-(0) = -  \limsup_{t \to \infty} \frac{\ln \|e^{tT}\psi\|_{{\mathcal{H}}}^2}{\ln  t} \,,
\end{equation*}
where $\mu_\psi^T$ is the spectral measure  of~$T$ associated with the vector~$\psi$.
\end{proposition}

Note that Proposition~\ref{decpolproposition} indicates that the power-law decaying rates of a semigroup orbit $(e^{tT}\psi)_{t \geq 0}$ may depend on sequences of time going to infinity; i.e., if $ d_{\mu_\psi^{T}}^-(0) < d_{\mu_\psi^{T}}^+(0)$ (see \cite{ACdeOscalesemigroup} for more details). 

Stimulated by results due to M\"uller and Tomilov in~\cite{Muller}, our main goal here is to obtain orbits of self-adjoint semigroups and unitary groups (in this case, for the (time-average)  return probability) that converge slowly to zero. More precisely, by exploring local dimensional properties of self-adjoint operators, we show explicitly how it is possible to perturb initial conditions, or generators, to obtain orbits of self-adjoint semigroups that converge to zero arbitrarily slowly, at least for a sequence of time going to infinity (Theorem~\ref{maintheorem1}). We also obtain a result about slow power-law decaying rates of the  return probability (see definition ahead) of  unitary evolution groups with purely continuous spectrum (Theorem~\ref{maintheorem3}). As an application of the arguments developed here, we compute (Baire) generically the local dimensions of systems with purely continuous spectrum (Theorem \ref{localtheorem}) to show that  the time-average (quantum) return probability, of (Baire) generic states of systems with purely absolutely continuous spectrum, has an oscillating behavior between a (maximum) fast power-law decay and a (minimum) slow power-law decay (Theorem \ref{maintheorem}); we note that such phenomenon has been found by the first three authors for several systems with singular continuous   spectra \cite{Aloisio,CarvalhoCorrelation} (see also \cite{OliveiraPAMS}).

Some words about notation: $T$ will always denote a self-adjoint operator acting in a complex and separable Hilbert space~$\mathcal H$; we denote its spectrum by $\sigma(T)$ and $\mu_\psi^T$ represents the spectral measure  of~$T$ associated with the vector $\psi\in\mathcal H$; for each Borel set $\Lambda \subset \mathbb{R}$, $P^T(\Lambda)$ represents the spectral resolution of $T$ over $\Lambda$; by $\mu$ we always mean a finite nonnegative Borel measure on~$\mathbb{R}$. 

The paper is organized as follows. The main results are described in Subsection~\ref{secslow}, which are stated in Theorems~\ref{maintheorem1}, \ref{maintheorem3} and~\ref{maintheorem}, along with some examples and dynamical consequences. In Section~\ref{proofs}, we present the proofs of Theorems~\ref{maintheorem1} and~\ref{maintheorem3}. In Section~\ref{localsec}, we study the local spectral properties of self-adjoint operators and conclude with a proof of Theorem~\ref{maintheorem}. 

\subsection{Statement of main results}\label{secslow}

\noindent Let $T$ be a pure point negative self-adjoint operator and  $(\eta_n)_{n\geq 1}$  the normalized eigenvectors of~$T$, say $T\eta_n = \lambda_n \eta_n$, so that $(\lambda_n)_{n \geq 1} \subset (-\infty,0)$ are the corresponding eigenvalues witch satisfy $\displaystyle\limsup_{n \to \infty} \lambda_n = 0.$
For  $\psi = \displaystyle\sum_{j=1}^N b_j \eta_j\in\mathcal{H}$, one has
\[\|e^{tT}\psi\|_{\mathcal{H}} =  \biggr\|\sum_{j=1}^N b_j e^{t\lambda_j} \eta_j \biggr\|_{\mathcal{H}} \leq N \max_{1\leq j \leq N} |b_j|e^{\lambda t},\]
with $\lambda = \displaystyle\max_{1\leq j \leq N} \lambda_j <0$, that is, for these initial conditions, the orbits vanish exponentially. Due to the abstract results by M\"uller and Tomilov~\cite{Muller} (see also \eqref{contexeq1}), given $\beta: \mathbb{R}\rightarrow (0, \infty)$ with
\begin{equation*}
    \lim_{t\rightarrow \infty}\beta(t)=\infty, 
\end{equation*} 
there exists $\psi \in \mathcal{H}$ such that 
\begin{equation*}
\limsup_{t\rightarrow\infty}\beta(t)\|e^{tT}\psi\|_{\mathcal{H}}=\infty,
\end{equation*}
 since $0 \in \sigma(T)$ in this case. In this specific context, in Theorem~\ref{maintheorem1}~i), we refine this result in the following sense: we show how it is possible to perturb any initial condition to explicitly display (in terms of the spectral structure of the generator) a new initial condition whose orbit  vanishes slower than any prescribed speed, at least for a sequence of time going to infinity, and, in item~ii), we  obtain a version of such result in terms of perturbations of the infinitesimal generator. 

%

\begin{theorem}\label{maintheorem1} Let $\beta: \mathbb{R}\rightarrow (0, \infty)$  be  a strictly increasing onto function, so
\begin{equation*}
    \lim_{t\rightarrow \infty}\beta(t)=\infty\,.
\end{equation*} 
\begin{enumerate}
\item[{\rm i)}] If $T$ is a  pure point negative self-adjoint operator as above, given $\psi\in\mathcal H$, there exists a sequence $(\psi_k)\subset\mathcal H$ that converges to~$\psi$ such that, for all~$k$,
\[
\limsup_{t\rightarrow\infty}\beta(t)\|e^{tT}\psi_k\|_{\mathcal{H}}=\infty\,.
\] 
\item[{ \rm ii)}]If $T$ is a negative bounded self-adjoint operator, then, for every nonzero $\psi \in \mathcal{H}$, there exists a sequence $(T_k)$ of negative bounded pure point self-adjoint operators that strongly converges to $T$ such that, for all $k$,
\[\limsup_{t\rightarrow \infty}\beta(t)\|e^{tT_k}\psi\|_{\mathcal{H}}=\infty\,.
\]
\end{enumerate}
\end{theorem}

\begin{remark}
\end{remark}
\rm{
\begin{enumerate}
\item [i)] Let us describe the vectors $\psi_k$ in the statement of Theorem \ref{maintheorem1}~i). Write  $\psi=\sum_{l=1}^{\infty}b_l\eta_l$ and, for each subsequence $(\lambda_{j_{l}})$ of eigenvalues of $T$ with $\lambda_{j_l}\uparrow 0$ and $\sum_{l=1}^{\infty} \frac{1}{\beta(1/|\lambda_{j_l}|)}< \infty$, one may pick
\begin{equation*}
\psi_k=\sum_{l=1}^{k}b_l\eta_l+\sum_{l=k+1}^{\infty} \frac{1}{\sqrt{\beta(1/|\lambda_{j_l}|)}}\, \eta_{j_{l}}\,.
\end{equation*}

\item[ii)] For every $\psi \in \mathcal{H}$, by the Spectral Theorem and dominated convergence,  
\[\lim_{t \to \infty} \|e^{tT}\psi\|_{\mathcal{H}}^2 = \mu_\psi^T(\{0\}) + \lim_{t \to \infty} \int_{\mathbb{R}_{-}\setminus \{0\}} e^{2tx} {\mathrm d}\mu_\psi^T(x) = \mu_\psi^T(\{0\}) = \|P^T(\{0\})\psi\|_{\mathcal{H}}^2\,.
\]
 Therefore, $e^{tT}$ is stable (i.e, all the orbits vanish as $t \rightarrow \infty$) if and only if zero is not an eigenvalue of~$T$. Hence,  Theorem~\ref{maintheorem1}~i) is particularly interesting in this case. Note that a well-known example of injective operator that satisfies the hypotheses of this theorem is the Hydrogen atom model restricted to its point subspace; see Chapter~11 in~\cite{Oliveira} for details.

\item[iii)] If $0 \not\in \sigma(T)$, then there exists $\gamma>0$ such that for each $\psi \in \mathcal{H}$ and each $t>0$, by the Spectral Theorem,  
\[\|e^{tT}\psi\|_{\mathcal{H}}^2 =  \int_{-\infty}^{-\gamma} e^{2tx} {\mathrm d}\mu_\psi^T(x) \leq e^{-2\gamma t} \|\psi\|_{\mathcal{H}}^2,\]
that is, all  orbits vanish exponentially as $t \rightarrow \infty$.

\item[iv)] Given any  nonzero initial condition $\psi \in \mathcal{H}$,  Theorem \ref{maintheorem1} ii)  says that we may always (strongly)  perturb the negative bounded self-adjoint infinitesimal generator~$T$ so that the orbit of~$\psi$  vanishes  slower than a prescribed speed~$\beta(t)$, at least for a sequence of time going to infinity.
\end{enumerate}}

Now we pass to unitary evolution groups. Given a self-adjoint operator~$T$  in $\mathcal{H}$, recall that  ${\mathbb{R}} \ni t \mapsto e^{-itT}$ is a one-parameter strongly continuous unitary evolution group and, for each $\psi\in \mathcal{H}$, $(e^{-itT}\psi)_{t \in \mathbb{R}}$ is the unique solution to the Schr\"odinger equation
\[
\begin{cases} \partial_t \psi = -iT\psi, \quad t \in {\mathbb{R}}, \\ \psi(0) = \psi\in {\rm dom} \, T.  \end{cases}
\]
A standard dynamical quantity that  probes the large time behavior of $e^{-itT}\psi$, and important in quantum mechanics, is the  so-called (time-average) {\em  return probability}, that is, 
\begin{equation*}
W_\psi^T(t) :=\frac{1}{t}\int_0^t |\langle e^{-isT} \psi,  \psi \rangle|^2 \, {\mathrm d}s.         
\end{equation*}
   
By the Spectral Theorem and  Wiener's Lemma~\cite{Oliveira}, 
 \begin{equation*}
\lim_{t \to \infty} W_\psi^T(t) = \sum_{\lambda \in \mathbb{R}} |\mu_\psi^T(\{\lambda\})|^2;         
\end{equation*}
in particular, if $T$ has purely continuous spectrum, then
 \begin{equation*}
\lim_{t \to \infty} W_\psi^T(t) =0.         
\end{equation*}
Our next result, Theorem~\ref{maintheorem3}, ensures   the existence of orbits, under each spectrally continuous unitary evolution group, with arbitrarily slow power-law  convergence rates. 

\begin{theorem}\label{maintheorem3} Let $T$ be a self-adjoint operator with purely continuous spectrum.
Then, there exists a vector $\psi \in \mathcal{H}$ such that, for every $\epsilon>0$,  
\[
\limsup_{t\rightarrow \infty}  \, t^\epsilon \, W_\psi^T(t)=\infty.
\]
\end{theorem}

\begin{remark}{\rm Although the existence of orbits  of operator semigroups that slowly decay  is a subject extensively studied in the literature, to the best knowledge of the present authors,  Theorem~\ref{maintheorem3}  is the first general result on slow dynamics for (spectrally continuous) unitary evolution groups.}
\end{remark}

\subsubsection{Absolutely continuous spectrum and generic quantum dynamics}
    
\noindent As mentioned before, we present an application to quantum dynamics of systems with purely absolutely continuous spectrum.     
    
\begin{theorem}\label{maintheorem} Let~$T$ be a bounded self-adjoint operator with purely absolutely continuous spectrum. Then,  the set of  $\psi \in \mathcal{H}$ such that, for all $k \in \mathbb{N}$,   
 \[\liminf_{t \to \infty}  t^{1-1/k}W_\psi^T(t) = 0 \quad and \quad  \limsup_{t \to \infty} t^{1/k}W_\psi^T(t)=\infty\] 
is generic in $\mathcal{H}$, i.e., it contains a dense $G_\delta$ subset of $\mathcal{H}$.
\end{theorem}

{\bf Application to Almost Mathieu Operator.} Recall that  the Almost Mathieu Operator $H^{\lambda,\alpha}_{\omega} $ is defined on~$\ell^2(\mathbb{Z})$ and given by 
\begin{equation}\label{lp11}
(H^{\lambda,\alpha}_{\omega} u)_n = u_{n+1} + u_{n-1} + 2\lambda \cos(2\pi(\omega+n\alpha)) u_n,
\end{equation}
with $\alpha$ and $\omega$ in $\mathbb{T} = \mathbb{R}\setminus\mathbb{Z}$, $\alpha$ irrational.
For each $w \in \mathbb{T}$, $\alpha \in \mathbb{T}$ irrational and $0<\lambda_1<\lambda_2<1$, let $X^{\lambda_{1,2},\alpha}_{\omega}$ be the set of operators $H^{\lambda,\alpha}_{\omega}$ with $\lambda \in [\lambda_1,\lambda_2]$, endowed with the metric
\[d(H^{\lambda}_{\omega,\alpha},H^{\lambda'}_{\omega,\alpha}) = |\lambda - \lambda'|.\]
It is well known that every $ H^{\lambda}_{\omega,\alpha} \in X^{\lambda_{1,2},\alpha}_{\omega}$ has purely absolutely continuous spectrum. For more details, see  \cite{DamanikDDP}. 
  
\begin{theorem}\label{theoremp} There exists a generic set $\mathcal{M}\subset\ell^2(\mathbb{Z})$ so that, for each $\psi \in \mathcal{M}$, the set of operators $H=H^{\lambda}_{\omega,\alpha} \in X^{\lambda_{1,2},\alpha}_{\omega}$ such that, for each $k \geq 1$,
\[\liminf_{t \to \infty}  t^{1-1/k}\,W_\psi^H(t)= 0 \quad {and} \quad  \limsup_{t \to \infty}  t^{1/k}\, W_\psi^H(t) = \infty\]
is generic in~$X^{\lambda_{1,2},\alpha}_{\omega}$.
\end{theorem}

\begin{remark}{\rm It is worth underlying that such phenomenon has been shown, for some singular continuous systems, by the first three authors  in~\cite{Aloisio,CarvalhoCorrelation} (see also \cite{OliveiraPAMS}) through the density of pure point operators in appropriate spaces, a quite different setting from this work. In this case of purely absolutely continuous spectrum this phenomenon is, in some sense, the counterpart of the  situation of an operator with pure point spectrum and quasiballistic transport \cite{ACdeOLatters,delRio}.}
\end{remark}


\begin{proof}[{Proof} {\rm (Theorem~\ref{theoremp})}]   The result is a direct consequence of Theorem~\ref{maintheorem} and an argument involving separability by some of the present authors in \cite{Aloisio,ACdeOLatters}. Since $X^{\lambda_{1,2},\alpha}_{\omega}$ is a separable space, let $(H_j)_j$ be a dense sequence in  $X^{\lambda_{1,2},\alpha}_{\omega}$. If $\mu^{j}_{\psi}$ denotes the spectral measure of the pair $(H_j,\psi)$, it follows from Theorem \ref{maintheorem} that   
\[
\mathcal{M}=\displaystyle\bigcap_{j} \big\{\psi \in \ell^2(\mathbb{Z}) \mid  \liminf_{t \to \infty}  t^{1-1/k} \, W_\psi^{H_j}(t)= 0 \quad {and} \quad  \limsup_{t \to \infty}  t^{1/k}  \, W_\psi^{H_j}(t) = \infty \big\}
\]
is generic in $\ell^2(\mathbb{Z})$. Since, for every $\psi \in \mathcal{M}$,
\[
\big\{H \mid  \liminf_{t \to \infty}  t^{1-1/k} \, W_\psi^{H}(t)= 0 \quad {and} \quad  \limsup_{t \to \infty}  t^{1/k}  \, W_\psi^{H}(t) = \infty\} \supset \{H_j\}
\]
is a  $G_\delta$ set in~$X^{\lambda_{1,2},\alpha}_{\omega}$ (see the proof of Proposition~2.2 in \cite{Aloisio}),  the result follows.
\end{proof}


\section{Proofs of Theorems \ref{maintheorem1} and \ref{maintheorem3}}\label{proofs}

\subsection{Proof of Theorem \ref{maintheorem1}}

\noindent The main ingredient in the proof of this theorem is the well-known expression of the spectral resolution of a pure point self-adjoint operator. Namely, since $T$ is a pure point operator, each $\psi \in \mathcal{H}$ can be written as
$\sum_{j=1}^\infty b_j \eta_j$ for some square-summable sequence $(b_j)_j$ of complex numbers, and the corresponding spectral measure is $\mu_\psi^T (\Lambda) = \sum_{\lambda_j \in \Lambda} |b_j|^2 \delta_{\lambda_j}$.

i) Let $(\lambda_{j_l})_l$ be a subsequence of eigenvalues of $T$, corresponding to orthonormal
eigenvectors $(\eta_{j_l})_l$, so that $\lambda_{j_l}\uparrow 0$ and $\displaystyle{\sum_{l= 1}^\infty \frac{1}{\beta(1/|\lambda_{j_l}|)}<\infty}$. Given $\psi \in \mathcal{H}$, write $\displaystyle{\psi= \sum_{j= 1}^\infty b_j\eta_j}$ and set 
\begin{equation*}
    \psi_k:=\sum_{l=1}^{k}b_l\eta_l+\sum_{l=k+1}^{\infty} \frac{1}{\sqrt{\beta(1/|\lambda_{j_l}|)}}\,\eta_{j_{l}}.
\end{equation*}
It follow that, for all $l\geq k+1$,
\begin{equation*}
    \mu^{T}_{\psi_k}([\lambda_{j_{l}},0])\geq \mu^{T}_{\psi_k}(\{\lambda_{j_{l}}\}) =  \frac{1}{\beta(1/|\lambda_{j_l}|)},
\end{equation*}
and therefore 
\begin{eqnarray*}
    \beta(1/|\lambda_{j_l}|)\,\|e^{(1/|\lambda_{j_l}|) T}\psi_{k}\|_{\mathcal{H}}&=&  \beta(1/|\lambda_{j_l}|)\left(~\int_{-\infty}^0 e^{2 (1/|\lambda_{j_l}|)x} {\mathrm d}\mu^{T}_{\psi_k}(x)\right)^{1/2}\\
     &\geq&  \beta(1/|\lambda_{j_l}|)\left(~\int_{\lambda_{j_l}}^0 e^{2 (1/|\lambda_{j_l}|)x} {\mathrm d}\mu^{T}_{\psi_k}(x)\right)^{1/2}\\
    &\geq&  e^{-1} \beta(1/|\lambda_{j_l}|) \,  (\mu^{T}_{\psi_k}([\lambda_{j_l},0]))^{1/2}\\
    &\geq&  e^{-1}  \sqrt{\beta(1/|\lambda_{j_l}|)}\,,
\end{eqnarray*}
which implies
\begin{equation*}
    \limsup_{t\rightarrow \infty}\beta(t)\|e^{tT}\psi_{k}\|_{\mathcal{H}}=\infty\,.
\end{equation*}

\ 

ii) Let $\psi \in \mathcal{H}$ and  $\{e_j\}_{j \geq 1}$ be an orthonormal basis of $\mathcal{H}$ such that $\psi=\displaystyle{\sum_{j=1}^{\infty}a_je_j}$ with $a_j \neq 0$ for infinitely many $j$s. Let $(a_{j_l})_{l \geq 1}$ be a subsequence of $(a_j)_{j \geq 1}$ with $|a_{j_l}|\downarrow 0$. Consider a positive sequence $t_l\rightarrow \infty$ so that $\beta(t_l)=|a_{j_l}|^{-2}$. 

For  each $k \geq 1$, set
\begin{equation*}
    T_k:= TP_{l\leq  k} - \sum_{l=k+1}^{\infty}\frac{1}{t_l}\langle e_{j_l},\cdot\rangle e_{j_l},
\end{equation*}
where $P_{l\leq k}$ is the projection onto the subspace generated by $\{e_{l}\}_{l\leq k}$. It is clear that $T_k \to T$ as $k \to \infty$ in the strong sense. The operator $TP_{j\leq  k}$ is  pure point  and negative. Note that for large enough $l$, $T_k (e_{j_l}) = - \frac{1}{t_l}e_{j_l}$.

Fix $k$; for large enough $l$, one has 
\begin{equation*}
    \mu^{T_k}_{\psi}([-1/t_l,0])\geq \mu^{T_k}_{\psi}(\{-1/t_l\}) = |a_{j_l}|^2 = \frac{1}{\beta(t_l)},
\end{equation*}
and therefore 
\begin{eqnarray*}
    \beta(t_l)\|e^{t_lT_k}\psi\|_{\mathcal{H}}&=& \beta(t_l)\left(~\int_{-\infty}^0 e^{2t_lx} {\mathrm d}\mu^{T_k}_{\psi}(x)\right)^{1/2}\\
    &\geq& \beta(t_l)\left(~\int_{-1/t_l}^0 e^{2t_lx} {\mathrm d}\mu^{T_k}_{\psi}(x)\right)^{1/2}\\
    &\geq& e^{-1}\beta(t_l)(\mu^{T_k}_{\psi}([-1/t_l,0]))^{1/2}\\
    &\geq& e^{-1}\sqrt{\beta(t_l)},
\end{eqnarray*}
which results in
\begin{equation*}
    \limsup_{t\rightarrow \infty}\beta(t)\|e^{tT_k}\psi\|_{\mathcal{H}}=\infty\,.
\end{equation*}
 
\subsection{Proof of Theorem \ref{maintheorem3}}

\noindent Let $\alpha \in [0,1]$. Recall that a finite positive Borel measure~$\mu$ on~$\mathbb{R}$ is uniformly $\alpha$-H\"older continuous (denoted {\rm U}$\alpha${\rm H}) if there exists a constant $C>0$ such that for each interval~$I$ with $\ell(I) < 1$, $\mu(I) \le C\, \ell(I)^\alpha$; here, $\ell(\cdot)$ denotes the Lebesgue measure on~$\mathbb{R}$. Theorem \ref{Strichartztheorem}~i) is, indeed, a particular case of a well-known theorem by Strichartz~\cite{strichartz1990}.

\begin{theorem}[Theorems 2.5 and 3.1  in \cite{Last}]\label{Strichartztheorem} Let $\mu$ be a finite Borel measure on $\mathbb{R}$ and $\alpha \in [0,1]$.

\begin{enumerate}

\item[\rm{i)}] If $\mu$ is {\rm U}$\alpha${\rm H}, then there exists  $C_\mu> 0$, depending only on $\mu$, such that for every $f \in {\mathrm L}^2(\mathbb{R}, d\mu )$ and every $t>0$, 
\[\frac{1}{t} \int_0^t \bigg|\int_{\mathbb{R}} e^{- isx} f(x)\,  {\mathrm d}\mu(x) \bigg|^2  {\mathrm d}s < C_{\mu} \|f\|_{{\mathrm L}^2(\mathbb{R}, d\mu )}^2 \, t^{-\alpha}. \] 

\item[\rm{ii)}] If there exists $C_{\mu}>0$ such that for every $t>0$, 
\[\frac{1}{t} \int_0^t \bigg|\int_{\mathbb{R}} e^{- isx}\,  {\mathrm d}\mu(x) \bigg|^2  {\mathrm d}s < C_{\mu} \, t^{-\alpha},\] 
then $\mu$ is {\rm U}$\frac{\alpha}{2}${\rm H}.

\end{enumerate}
\end{theorem}

\begin{lemma}[Lemma  2.1 in \cite{ACdeOscalesemigroup}]\label{normallema}Let~$T$ be a negative self-adjoint operator with  $0 \in \sigma(T)$ and let  $\alpha:{\mathbb{R}}_+ \longrightarrow (0,\infty)$ be such that
\[\lim_{t \to \infty} \alpha(t) = \infty.\]
Then, there exist $\eta \in \mathcal{H}$ and a sequence $t_j \rightarrow \infty$ such that, for sufficiently large~$j$,  
\[\mu_{\eta}^T\big(B(0;{1/t_j})\big) \geq \frac{1}{\alpha(t_j)}.\]
\end{lemma}

\begin{proof}[{Proof} {\rm (Theorem~\ref{maintheorem3})}] Let $x \in \sigma(T)$ and set $L_x=(-\infty,x]\cap \sigma(T)$,  $T_x=T P^{T}(L_x)$  and $T^{0}_x=T_x-xI$. So, by Lemma  \ref{normallema}, there exist  $\psi \in \mathcal{H}$ and $\varepsilon_j  \rightarrow 0$  such that, for sufficiently large $j$, 
\begin{eqnarray*}
& & \mu^{T^{0}_x}_{\psi}(B(0;\varepsilon_j)) \geq \frac{1}{-\ln(\varepsilon_j)}\; \Rightarrow \; \mu^{T_x}_{\psi}(B(x;\varepsilon_j))\geq \frac{1}{-\ln(\varepsilon_j)}
\\ &\quad&   \Rightarrow\;  \mu^{T}_{\psi}(B(x;\varepsilon_j)) \geq \mu^{T}_{\psi}(B(x;\varepsilon_j) \cap L_x) =  \mu^{T_x}_{\psi}(B(x;\varepsilon_j)) \geq \frac{1}{-\ln(\varepsilon_j)}.
\end{eqnarray*}  
Hence, $\mu^{T}_{\psi}$ is not U$\alpha$H for all $0<\alpha\leq 1$. Thus, by Theorem \ref{Strichartztheorem}, for every $\epsilon>0$,
\[\limsup_{t \to \infty} t^\epsilon \,\frac{1}{t} \int_0^t \bigg|\int_{\mathbb{R}} e^{- isx}\,  {\mathrm d}\mu^{T}_{\psi}(x) \bigg|^2  {\mathrm d}s = \infty. \] 
Since, by Spectral Theorem, for every $s \in \mathbb{R}$ 
\[\langle e^{-isT}\psi, \psi \rangle = \int_{\mathbb{R}} e^{- isx}\,  {\mathrm d}\mu^{T}_{\psi}(x),\]
the result follows. 
\end{proof}


\section{Generic spectral properties and proof of Theorem \ref{maintheorem}}\label{localsec}

\noindent As mentioned in the Introduction, now  we compute (Baire) generically the local dimensions of systems with purely continuous spectrum (Theorem \ref{localtheorem}) in order to prove  Theorem~\ref{maintheorem}.

Note that, for each $x \in \mathbb{R}$ and each $\epsilon > 0$, 
\[ \int_{\mathbb{R}} e^{-2t|x-y|}{\mathrm d}\mu(y)  \geq \int_{B(x;{1/t})} e^{-2t|x-y|}{\mathrm d}\mu(y)  \geq e^{-2} \mu(B(x;{1/t})).\] 
On the other hand, for  each $0 < \delta < 1$ and each  $t>0$,
\begin{eqnarray}\label{exploc3}
\nonumber \!\!\! \int_{\mathbb{R}} e^{-2t|x-y|}{\mathrm d}\mu(y) &=& \int_{B(x;\frac{1}{t^{1-\delta}})}  e^{-2t|x-y|}{\mathrm d}\mu(y) +   \int_{B(x;\frac{1}{t^{1-\delta}})^c}  e^{-2t|x-y|}{\mathrm d}\mu(y)\\
&\leq&  \mu\big(B\big(x,{1/t^{1-\delta}})) +   e^{-t^\delta} \mu(\mathbb{R}).
\end{eqnarray}

Thus, at least when $\mu$ has a certain local regularity (with respect to the Lebesgue measure), we expect that $\int_{\mathbb{R}} e^{-2t|x-y|}{\mathrm d}\mu(y)$ and $\mu(B(x;{1/t}))$ are asymptotically comparable as $t\to\infty$. 
In this sense, the following identities are expected:
\begin{equation}\label{exploc1}
\liminf_{t \to \infty} \frac{\ln [ \int_{\mathbb{R}} e^{-2t|w-y|}{\mathrm d}\mu(y)]}{\ln t} = -d_\mu^+(w),
\end{equation}
\begin{equation}\label{exploc2}
\limsup_{t \to \infty} \frac{\ln [ \int_{\mathbb{R}} e^{-2t|w-y|}{\mathrm d}\mu(y)]}{\ln  t} = -d_\mu^-(w).
\end{equation}
Indeed, these identities were proven in \cite{ACdeOscalesemigroup} (note that  since it is not possible to compare directly the  two terms on the right-hand side of~\eqref{exploc3}, some caution should be exercised when checking~(\ref{exploc1}) and~(\ref{exploc2})). We use such identities in the proof of Theorem~\ref{localtheorem} below.

\begin{theorem}\label{localtheorem}
Let~$T$ be a bounded self-adjoint operator with purely continuous spectrum. Then, there exists a generic set $\mathcal{M} \subset \mathcal{H}$  such that for each $\psi \in \mathcal{M}$, the set 
\begin{equation*}
{\mathcal{J}}_{\psi}:= \big\{x \in \sigma(T)  \mid d^-_{\mu_\psi^T}(x) =0 \, \, \, {\rm and }\, \, \,  d^+_{\mu_\psi^T}(x) = \infty \big\}
\end{equation*}
is generic in $\sigma(T)$.
\end{theorem}

\begin{remark} { \rm Theorem \ref{localtheorem} is particularly interesting when~$T$ has purely absolutely continuous spectrum, since it shows the striking difference between the typical behaviour of $d^\pm_{\mu_\psi^T}$ from the topological and measure points of view; namely, if $\mu_\psi^T$ is purely absolutely continuous, then it is well known that  $\mu_\psi^T$-$\essinf d_{\mu_\psi^T}^- =1$ (see \cite{Falconer} for details).} 
\end{remark}

\begin{proof}[{Proof} {\rm (Theorem~\ref{localtheorem})}] Note that is enough to show that, for each $x\in\sigma(T)$, the set 
\begin{equation}\label{mainprof1}
{\mathcal{G}}(x) := \big\{\psi\in \mathcal H \mid d^-_{\mu_\psi^T}(x)=0 \ \ \text{and} \ \ d^{+}_{\mu_\psi^T}(x)=\infty\big\}
\end{equation}
is generic in $\mathcal{H}$. Namely, given $0\ne\psi \in \mathcal{H}$, since, by dominated convergence, for every $t>0$ the mapping 
\[\sigma(T) \ni x \mapsto \int_{\mathbb{R}} e^{-2t|x-y|}{\mathrm d}\mu_\psi^T(y)\]
is continuous, it follows that 
\[
{\mathcal{J}}_{\psi}\,=\; A_-\cap A_+
\] is a $G_\delta$ set in $\sigma(T)$, given that 
\[
A_-=\bigcap_{l\geq 1}  \bigcap_{n\geq 1} \bigcap_{k \geq 1} \bigcup_{t\geq k} \Big\{x \in \sigma(T) \mid  t^l \int_{\mathbb{R}} e^{-2t|x-y|}{\mathrm d}\mu_\psi^T(y) < \frac{1}{n} \Big\}
\] and
\[
A_+ = \bigcap_{l\geq 1}   \bigcap_{n\geq 1} \bigcap_{k \geq 1} \bigcup_{t\geq k} \Big\{x \in \sigma(T) \mid  t^{{1}/{l}} \int_{\mathbb{R}} e^{-2t|x-y|}{\mathrm d}\mu_\psi^T(y) > n \Big\}.
\]

 Now, let $(x_n)_{n \in \mathbb{N}} \subset \sigma(T)$ be a  dense sequence in $\sigma(T)$. So, if
 
 \[\psi \in \mathcal{M} = \bigcap_n {\mathcal G}(x_n)= \big \{\varphi \mid d^-_{\mu_\varphi^T}(x_n)=0 {\rm \, \, and \, \,}  d^{+}_{\mu_\varphi^T}(x_n)=\infty, {\rm \, \, for \, \, each \, \,} n \in \mathbb{N} \big\},\]
it follows that ${\mathcal{J}}_{\psi}$ is generic in $\sigma(T)$.

After such preliminaries,    we divide the proof of Theorem~\ref{localtheorem} in 4 steps.

\

\noindent {\bf Step 1.} Let us show that for each $\rho > 0$ and each $x \in \sigma(T)$,  
\begin{equation*}
    \big\{\psi\in \mathcal{H} \mid d^{+}_{\mu^{T}_{\psi}}(x)\geq d^{-}_{\mu^{T}_{\psi}}(x)\geq \rho \big\}
\end{equation*}
is dense in $\mathcal{H}$. Namely, let for each $n \in\mathbb{N}$ and each $y\in\mathbb{R}$, 
\begin{equation*}
f_{n,\rho}(x,y):= \left(1-e^{-n|x-y|^{\rho}}\right)^{1/2}, 
\end{equation*}
and for each $\psi\neq 0$, let $\psi_{n}:=f_{n,\rho}(x,T)\psi$, where $f_{n,\rho}(x,T) := P^T(f_{n,\rho}(x,\cdot))$. Since $\mu^{T}_{\psi}$ is purely continuous, one gets, by the Spectral Theorem and dominated convergence, that  
\begin{eqnarray*}
\nonumber\|\psi_n-\psi\|^{2} &=& \|f_{n,\rho}(x,T)\psi-\psi\|^{2}\\
\nonumber&=& \|(f_{n,\rho}(x,T)-1)\psi\|^{2}\\
&=& \int_{\mathbb{R}}\left|\left(1-e^{-n|x-y|^{\rho}}\right)^{1/2}-1\right|^{2}{\mathrm d}\mu^{T}_{\psi}(y)\label{eq1}\\
\nonumber&=& \mu^{T}_{\psi}(\{x\}) + \int_{\mathbb{R}\setminus \{x\}}\left|\left(1-e^{-n|x-y|^{\rho}}\right)^{1/2}-1\right|^{2}{\mathrm d}\mu^{T}_{\psi}(y)\\
&=& \int_{\mathbb{R}\setminus \{x\}}\left|\left(1-e^{-n|x-y|^{\rho}}\right)^{1/2}-1\right|^{2}{\mathrm d}\mu^{T}_{\psi}(y)\label{eq2} \longrightarrow 0
\end{eqnarray*}
as $n \to \infty$, that is, $\psi_n \rightarrow \psi $ in $\mathcal{H}$. 

Now, by Fubini's Theorem,
\begin{eqnarray*}
\int_{\mathbb{R}}e^{-2t|x-y|}{\mathrm d}\mu^{T}_{\psi_n}(y)&=& \int_{\mathbb{R}}e^{-2t|x-y|}{\mathrm d}\mu^{T}_{f_{n,\rho}(x,T)\psi}(y)\\
&=& \int_{\mathbb{R}}e^{-2t|x-y|} |f_{n,\rho}(x,y)|^2{\mathrm d}\mu^{T}_{\psi}(y)\\
&=& \int_{\mathbb{R}}e^{-2t|x-y|} (1-e^{-n|x-y|^{\rho}}){\mathrm d}\mu^{T}_{\psi}(y)\\
&=& \int_{\mathbb{R}}e^{-2t|x-y|}|x-y|^{\rho} \frac{(1-e^{-n|x-y|^{\rho}})}{|x-y|^{\rho}}{\mathrm d}\mu^{T}_{\psi}(y)\\
&\leq & \frac{\rho}{2^\rho t^{\rho}} \int_{\mathbb{R}} \frac{(1-e^{-n|x-y|^{\rho}})}{|x-y|^{\rho}}{\mathrm d}\mu^{T}_{\psi}(y)\\
&=& \frac{\rho}{2^\rho  t^{\rho}} \int_{\mathbb{R}} \int_{0}^{n} e^{-s|x-y|^{\rho}} {\mathrm d}s\,{\mathrm d}\mu^{T}_{\psi}(y)\\
&\leq& \frac{n\rho}{2^\rho  t^{\rho}} \|\psi\|^{2}.
\end{eqnarray*}
Thus, it follows from identity (\ref{exploc2}) that for each $n \geq 1$, $d^{-}_{\mu^{T}_{\psi_n}}(x)\geq \rho$, and so 
\begin{equation*}
    \big\{\psi\in \mathcal{H} \mid d^{+}_{\mu^{T}_{\psi}}(x)\geq d^{-}_{\mu^{T}_{\psi}}(x)\geq \rho \big\}
\end{equation*}
is dense in $\mathcal{H}$. 

\ 

\noindent {\bf Step 2.}  Let us show that  for every $x\in \sigma (T)$,  there exists $\eta \in \mathcal{H}$  such that                       $d^{-}_{\mu^T_\eta}(x)=0$. 
Set $L_x=(-\infty,x]\cap \sigma(T)$,  $T_x=T P^{T}(L_x)$  and $T^{0}_x=T_x-xI$. So, by Lemma  \ref{normallema}, there exist  $\eta \in \mathcal{H}$ and $\varepsilon_j  \rightarrow 0$  such that for sufficiently large $j$, 

\begin{equation*}
\mu^{T^{0}_x}_{\eta}(B(0;\varepsilon_j))\geq \frac{1}{-\ln(\varepsilon_j)}\; \Rightarrow\; \mu^{T_x}_{\eta}(B(x;\varepsilon_j))\geq \frac{1}{-\ln(\varepsilon_j)}
\end{equation*}  

\begin{equation*}
\Rightarrow\; \ln\left(\mu^{T}_{\eta}(B(x;\varepsilon_j))\right) \geq  \ln\left(\mu^{T}_{\eta}(B(x;\varepsilon_j) \cap L_x)\right) = \ln\left(\mu^{T_x}_{\eta}(B(x;\varepsilon_j))\right)\geq \ln\left(\frac{1}{-\ln(\varepsilon_j)}\right)
\end{equation*} 

\begin{equation*}\label{eq3}
\Rightarrow \; \frac{\ln\left(\mu^{T}_{\eta}(B(x;\varepsilon_j))\right)}{\ln \varepsilon_j}\leq \frac{\ln\left(\frac{1}{-\ln(\varepsilon_j)}\right)}{\ln \varepsilon_j} \Rightarrow d^{-}_{\mu^{T}_{\eta}}(x) =  0.
\end{equation*} 
 
\ 

\noindent {\bf Step 3.} Let us show that for every $x \in \sigma(T)$, 
\begin{equation*}
    \big\{\psi\in \mathcal{H} \mid d^{-}_{\mu^{T}_{\psi}}(x)=0 \big\}
\end{equation*}
is dense in $\mathcal{H}$. Namely,  let  $x\in \sigma (T)$  and  set, for every $n \geq 1$, 
\begin{equation*}
S_n :=\left(-\infty,x-\frac{1}{n}\right)\cup \{x\}\cup \left(x+\frac{1}{n}, \infty\right).
\end{equation*}
Set also, for each $\psi \in \mathcal{H}$ and each $n \geq 1$,
\begin{equation*}
\psi_n:=P^{T}(S_n)\psi + \frac{1}{n}\eta,
\end{equation*}
where $\eta$  is given by {\bf Step 2}. One has that $\psi_n\rightarrow \psi$ in  $\mathcal{H}$, since $P^{T}(S_n)\rightarrow\mathbf{1}$ in the strong sense. Moreover, for each $n \geq 1$ and  each $0<\varepsilon<\frac{1}{n}$,  one has 
\begin{eqnarray*}
\mu^{T}_{\psi_n}(B(x;\varepsilon))&=& \langle P^{T}(B(x;\varepsilon))\psi_n, \psi_n \rangle\\
&=& \langle P^{T}(B(x;\varepsilon))P^{T}(S_n)\psi, \psi_n \rangle+\frac{1}{n}\langle P^{T}(B(x;\varepsilon))\eta, \psi_n \rangle\\
&=& \langle P^{T}(B(x;\varepsilon)\cap S_n)\psi, \psi_n \rangle+\frac{1}{n}\langle P^{T}(B(x;\varepsilon))\eta, \psi_n \rangle\\
&=& \langle P^{T}(\{x\})\psi, \psi_n \rangle+\frac{1}{n}\langle P^{T}(B(x;\varepsilon))\eta,  P^{T}(S_n)\psi \rangle+\frac{1}{n^2}\langle P^{T}(B(x;\varepsilon))\eta,   \eta\rangle\\
&=& \langle P^{T}(\{x\})\psi, \psi_n \rangle+\frac{1}{n}\langle P^{T}(\{x\})\eta,  \psi \rangle+\frac{1}{n^2}\langle P^{T}(B(x;\varepsilon))\eta,   \eta\rangle\\
&=& \frac{1}{n^2}\langle P^{T}(B(x;\varepsilon))\eta, \eta\rangle\\
&=& \frac{1}{n^2}\,\mu^{T}_{\eta}(B(x;\varepsilon))\,,
\end{eqnarray*}
and so
\begin{equation*}
d^{-}_{\psi^{T}_{\psi_n}}(x)=\liminf_{\varepsilon\downarrow 0}\frac{\ln(\mu^{T}_{\psi_n}(B(x;\varepsilon)))}{\ln \varepsilon}=\liminf_{\varepsilon\downarrow 0}\frac{\ln(\mu^{T}_{\eta}(B(x;\varepsilon))}{\ln \varepsilon}=d^{-}_{\psi^{T}_{\eta}}(x)=0.
\end{equation*}
Hence,
\begin{equation*}
  \big  \{\psi\in \mathcal{H} \mid d^{-}_{\mu^{T}_{\psi}}(x)=0 \big\}
\end{equation*}
is dense in $\mathcal{H}$.

\ 

\noindent {\bf Step 4.} Finally, in this step, we finish the proof of the theorem. Since, for each $x \in \mathbb{R}$ and each $t>0$, the mapping 
\[\
\mathcal{H} \ni \psi \mapsto \int_{\mathbb{R}} e^{-2t|x-y|}{\mathrm d}\mu_\psi^T(y) = \langle g_t(T,x)\psi, \psi \rangle\,,
\]
with $g_t(y,x) =  e^{-2t|x-y|}$, is continuous, it follows that for every $x \in \mathbb{R}$, each one of the sets
\[
B_-(x):=\big\{\psi  \in \mathcal{H}  \mid  d^+_{\mu_\psi^T}(x) = \infty\big\}=\bigcap_{l\geq 1}  \bigcap_{n\geq 1} \bigcap_{k \geq 1} \bigcup_{t\geq k} \Big\{\psi  \in \mathcal{H} \mid  t^l \int_{\mathbb{R}} e^{-2t|x-y|}{\mathrm d}\mu_\psi^T(y) < {1}/{n} \Big\}
\]
and
\[B_+(x):=\big\{\psi\in\mathcal{H}\mid d^-_{\mu_\psi^T}(x) =0\big\} = \bigcap_{l\geq 1}   \bigcap_{n\geq 1} \bigcap_{k \geq 1} \bigcup_{t\geq k}\Big \{\psi  \in \mathcal{H} \mid  t^{\frac{1}{l}} \int_{\mathbb{R}} e^{-2t|x-y|}{\mathrm d}\mu_\psi^T(y) > n \Big \}
\] is a $G_\delta$ set in $\mathcal{H}$. 
Thus, it follows from Steps~1.\ and~3.\ that for each $x \in \sigma(T)$, both $B_-(x)$ and $B_+(x)$ are generic sets in $\mathcal{H}$, and so
\begin{equation}\label{eqlocal1}
{\mathcal{G}}(x) = \bigcap_{n \geq 1}\big\{\psi\in\mathcal{H} \mid d^-_{\mu_\psi^T}(x)=0 \ \ \text{and} \ \ d^{+}_{\mu_\psi^T}(x) \geq n \big\}
\end{equation} 
is also generic in $\mathcal{H}$.
\end{proof}

\subsection{Proof Theorem  \ref{maintheorem}}

 \noindent We will also need the following result.

\

\noindent\textbf{Claim}. If $x\in\sigma(T)$, then ${\mathcal{G}}(x)  \subset \big\{\psi\in\mathcal{H} \mid \mu_\psi^T \, \, {\rm is \, not} \,  {\rm U}(1/k){\rm H}, \, \forall k \in \mathbb{N} \big\}$.  Indeed, it is enough to note that, given $\alpha>0$, if $\mu_\psi^T$ is ${\rm U}\alpha{\rm H}$, then for each $x\in\sigma(T)$, $d^{-}_{\mu_\psi^T}(x) \geq \alpha$. 

\

For $x\in\sigma(T)$, if $\psi \in {\mathcal{G}}(x)$,  then one has from the Claim that for each $k\ge 1$,  $\mu_\psi^T$   is  not  ${\rm U}(1/2k){\rm H}$. Thus, it follows from Theorem~\ref{Strichartztheorem} ii) that for each $k\ge 1$,   
\[\limsup_{t \to \infty}  t^{1/k}\,W_\psi^T(t) = \infty,\]
and then one has from the proof of Theorem~\ref{localtheorem} (recall (\ref{eqlocal1})) that for each $k\ge 1$, the set
\[
 \big\{\psi \in \mathcal{H} \mid \limsup_{t \to \infty}  t^{1/k}\,W_\psi^T(t) = \infty \big\} \supset {\mathcal{G}}(x) 
\] 
is generic in $\mathcal{H}$.

It remains to prove that for each $k \geq 1$, the set 
\[
A_k:=\big\{\psi \in \mathcal{H} \mid \liminf_{t \to \infty}  t^{1-1/k}\,W_\psi^T(t)  = 0 \big\}
\] 
is generic in $\mathcal{H}$. The proof that for each $k\ge 1$, $A_k$ is a $G_\delta$ subset of $\mathcal{H}$ follows closely the arguments presented in the proof of Theorem \ref{localtheorem}. On the other hand, it follows from Theorem~\ref{Strichartztheorem} i) that for each $k\ge 1$, 
\[\{\psi \in {\mathcal{H}}  \mid \mu_\psi^T \ {\rm is} \ {\rm U}1{\rm H}\} =: {\mathcal{H}}^T_{\mathrm{UH}}(1)  \subset A_k.\]
Finally, since by Theorem 5.2 in \cite{Last} (by taking $\alpha =1$), ${\mathcal{H}}^T_{\mathrm{UH}}(1)$ is dense in $\mathcal{H}$, it follows that for each $k\ge 1$, $A_k$ is a dense $G_\delta$ subset of $\mathcal{H}$ (recall that $T$ has purely absolutely continuous, by hypothesis).

\begin{center} \Large{Acknowledgments} 
\end{center}
\addcontentsline{toc}{section}{Acknowledgments}

 GS thanks the partial support by CAPES (Brazilian agency). SLC thanks the partial support by Fapemig (Minas Gerais state agency; Universal Project under contract 001/17/CEX-APQ-00352-17). CRdO thanks the partial support by CNPq (a Brazilian government agency, under contract 303689/2021-8). 


\

\noindent  Email: moacir.aloisio@ufvjm.edu.br, Departamento de Matem\'atica e Estatística, UFVJM, Diamantina, MG, 39100-000 Brazil

\noindent  Email: silas@mat.ufmg.br, Departamento de Matem\'atica, UFMG, Belo Horizonte, MG, 30161-970 Brazil

\noindent  Email: oliveira@dm.ufscar.br,  Departamento  de  Matem\'atica,   UFSCar, S\~ao Carlos, SP, 13560-970 Brazil

\noindent  Email: gesoares2017@gmail.com, Departamento de Matem\'atica, UFMG, Belo Horizonte, MG, 30161-970 Brazil

\end{document}